\documentclass{compositio}

\usepackage{graphicx}
\usepackage[font=small,labelfont=bf]{caption}
\usepackage{epstopdf}
\usepackage{enumitem}
\usepackage{lipsum}
\usepackage{comment}
\usepackage{amsthm}
\usepackage{amsmath,amsfonts,amsthm,bm}
\usepackage{verbatim}
\usepackage{tikz}
\usepackage{tikz-cd}
\usetikzlibrary{positioning}
\usetikzlibrary{arrows}
\usetikzlibrary{decorations.markings}
\usepackage{yhmath}
\usepackage{relsize}
\usepackage{calc}
\usepackage{mathrsfs}

\newtheorem{theorem}{Theorem}[section]
\makeatletter
\newcommand{\proofcase}[2]{%
  \par
  \addvspace{\medskipamount}%
  \noindent\emph{Case #1: #2}%
  \@afterheading
}
\makeatother
\makeatletter
\newcommand{\proofstep}[2]{%
  \par
  \addvspace{\medskipamount}%
  \noindent\emph{Step #1: #2}%
  \@afterheading
}
\makeatother
\newtheorem{lemma}[theorem]{Lemma}
\newtheorem{corollary}[theorem]{Corollary}

\theoremstyle{definition}

\theoremstyle{remark}

\numberwithin{equation}{section}



\makeatletter
\def\blfootnote{\xdef\@thefnmark{}\@footnotetext}
\makeatother

\begin{document}

\title{Proper CAT(0) Actions of Unipotent-Free Linear Groups}

\author{Sami Douba}
\thanks{The author was partially supported by the National Science Centre, Poland UMO-2018/30/M/ST1/00668.}
\email{sami.douba@mail.mcgill.ca}
\address{Department of Mathematics and Statistics, McGill University, Montreal, QC H3A 0B9}

\begin{abstract}
Let $\Gamma$ be a finitely generated group of matrices over $\mathbb{C}$. We construct an isometric action of $\Gamma$ on a complete CAT(0) space $X$ such that the restriction of this action to any subgroup of $\Gamma$ containing no nontrivial unipotent elements is well behaved. As an application, we show that if $M$ is a graph manifold that does not admit a nonpositively curved Riemannian metric, then any finite-dimensional $\mathbb{C}$-linear representation of $\pi_1(M)$ maps a nontrivial element of $\pi_1(M)$ to a unipotent matrix. In particular, the fundamental groups of such 3-manifolds do not admit any faithful finite-dimensional unitary representations.
\end{abstract}

\maketitle

\section{Introduction}

Let $F$ be a field and $n$ a positive integer. An element of $\mathrm{SL}_n(F)$ is {\it unipotent} if it has the same characteristic polynomial as the identity matrix. In \cite{button2017properties, button2019aspects}, Button demonstrated that finitely generated subgroups of $\mathrm{SL}_n(F)$ containing no infinite-order unipotent elements share some properties with groups acting properly by semisimple isometries on complete $\mathrm{CAT}(0)$ spaces. Indeed, Button showed that if $F$ has positive characteristic (in which case any unipotent element of~$\mathrm{SL}_n(F)$ has finite order), then any finitely generated subgroup of~$\mathrm{SL}_n(F)$ admits such an action \cite[Theorem~2.3]{button2019aspects}. The main theorem of this article is intended to serve as an analogue of the latter result in the characteristic-zero setting.

\begin{theorem}\label{main}
Let $\Gamma$ be a finitely generated subgroup of $\mathrm{SL}_n(\mathbb{C})$, $n> 0$. Then $\Gamma$ acts on a complete CAT(0) space $X$ such that 
\begin{enumerate}[label=(\roman*)]
\item for any subgroup $H < \Gamma$ containing no nontrivial unipotent matrices, the induced action of~$H$ on $X$ is proper;
\item if such a subgroup $H$ is free abelian of finite rank, then $H$ preserves and acts as a lattice of translations on a thick flat in $X$; in particular, any infinite-order element of such a subgroup~$H$ acts ballistically on $X$; \label{thickflat}
\item if $g \in \Gamma$ is a diagonalizable, then $g$ acts as a semisimple isometry of $X$.
\end{enumerate}
\end{theorem}

See Section \ref{prelim} for the relevant definitions. The space $X$ is a finite product of symmetric spaces of non-compact type and (possibly locally infinite) Euclidean buildings. Since an element of $\mathrm{SL}_n(\mathbb{C})$ that is both diagonalizable and unipotent must be trivial, the following corollary is immediate.

\begin{corollary}\label{maincor}
Any finitely generated subgroup of $\mathrm{SL}_n(\mathbb{C})$ consisting entirely of diagonalizable matrices acts properly by semisimple isometries on a complete $\mathrm{CAT}(0)$ space.
\end{corollary}

Precompact subgroups of $\mathrm{SL}_n(\mathbb{C})$ are conjugate into $\mathrm{SU}(n)$ and thus consist entirely of diagonalizable matrices. Furthermore, by the Peter–Weyl theorem, any compact Lie group can be realized as a compact subgroup of $\mathrm{SL}_n(\mathbb{C})$ for some $n$ \cite[Theorem~III.4.1]{brocker1985representations}. Thus, by Corollary \ref{maincor}, any finitely generated subgroup of a compact Lie group admits a proper action by semisimple isometries on a complete $\mathrm{CAT}(0)$ space. 

For us, a {\it graph manifold} is a connected closed orientable irreducible non-Seifert 3-manifold all of whose JSJ blocks are Seifert. Property \ref{thickflat} of the action described in Theorem \ref{main} allows us to conclude the following fact about representations of fundamental groups of graph manifolds.

\begin{theorem}\label{maingraph}
Let $M$ be a graph manifold and let $\rho: \pi_1(M) \rightarrow \mathrm{SL}_n(\mathbb{C})$ be any representation. If $M$ does not admit a nonpositively curved Riemannian metric, then there is a JSJ torus $S$ of~$M$ and a nontrivial element $h \in \pi_1(S) < \pi_1(M)$ such that $\rho(h)$ is unipotent. 
\end{theorem}

A manifold is said to be {\it nonpositively curved (NPC)} if it admits a nonpositively curved Riemannian metric. By work of Agol \cite{agol2013virtual}, Przytycki–Wise \cite{przytycki2018mixed}, and Liu \cite{liu2013virtual}, the fundamental group of any closed NPC 3-manifold virtually embeds in a finitely generated right-angled Artin group (RAAG). Moreover, Agol \cite{315430} showed that any finitely generated RAAG embeds in a compact Lie group. On the other hand, if $M$ is a closed aspherical non-NPC 3-manifold, then either $M$ is Seifert, in which case there is a nontrivial element of $\pi_1(M)$ that gets mapped to a virtually unipotent matrix under any faithful finite-dimensional linear representation of $\pi_1(M)$ (see, for example, the discussion in the introduction of \cite{douba2021virtually}), or the orientation cover of $M$ is a non-NPC graph manifold. Thus, we obtain from Theorem \ref{maingraph} the following corollary.  

\begin{corollary}
A closed aspherical 3-manifold $M$ is nonpositively curved if and only if there is a faithful finite-dimensional $\mathbb{C}$-linear representation of $\pi_1(M)$ whose image contains no nontrivial unipotent matrices. 
\end{corollary}

We remark that a result similar to Theorem \ref{main} was announced in \cite[Theorem~1.4]{matsnev2007baum}. However, the proof of \cite[Theorem~4.8]{matsnev2007baum}, on which that result rests, contains an error; a $\mathrm{CAT}(0)$ action of a finitely generated linear group $G$ with proper restrictions to certain subgroups of $G$ is desired, but what is provided is a proper $\mathrm{CAT}(0)$ action for each such subgroup of~$G$.

\subsection*{Organization}
In Section \ref{prelim}, we define the relevant objects, discuss briefly some properties of ballistic isometries of complete $\mathrm{CAT}(0)$ spaces, and introduce the central notion of a ``thick flat" in such a space. In Section \ref{lemmas}, we prove several lemmas used in the proofs of Theorems \ref{main} and \ref{maingraph}. The latter proofs are contained in Section \ref{proofs}. 

\subsection*{Acknowledgements}
I am deeply grateful to Piotr Przytycki for his encouragement, patience, and guidance. I also thank Bruno Duchesne, Zachary Munro, and Abdul Zalloum for helpful discussions.

\section{Preliminaries}\label{prelim}
\subsection{Complete $\mathrm{CAT}(0)$ spaces}
Let $X$ be a complete $\mathrm{CAT}(0)$ space and $\partial X$ its visual boundary. We will make references to the cone topology on $\overline{X} := X \cup \partial X$, described in \cite{bridson1999metric}. Under this topology, a sequence of points~$x_n \in X$ converges to $\xi \in \partial X$ if and only if for some (hence any) point $x_0 \in X$, the geodesics joining $x_0$ to $x_n$ converge uniformly on compact intervals to the unique geodesic ray emanating from $x_0$ in the class of $\xi$. In addition, we will use the angular metric $\angle$ on $\partial X$, also described in \cite{bridson1999metric}. Note that the topology on $\partial X$ induced by the angular metric is in general finer than the cone topology on $\partial X$. 

An $r$-dimensional {\it flat} in $X$ is an isometrically embedded copy of $\mathbb{R}^r$ in $X$. We say $X$ is {\it $\pi$-visible} if for any $\xi, \eta \in \partial X$ satisfying $\angle (\xi, \eta) = \pi$, there is a geodesic line in $X$ whose endpoints on $\partial X$ are $\xi$ and $\eta$. Since Euclidean spaces are  $\pi$-visible, a complete $\mathrm{CAT}(0)$ space $X$ with the property that any two points on $\partial X$ lie on the boundary of a common flat in $X$ is also $\pi$-visible. Note that if $X$ is a Euclidean building, a symmetric space of non-compact type, or a product of such spaces, then $X$ possesses the latter property by the building structure on $\partial X$, so that $X$ is $\pi$-visible. For more information on symmetric spaces, we refer the reader to the monograph \cite{eberlein1996geometry}. 

\subsection{Isometries of complete $\mathrm{CAT}(0)$ spaces}
Let $(X,d_X)$ be a complete $\mathrm{CAT}(0)$ space and let $g \in \mathrm{Isom}(X)$. The {\it translation length} of $g$ is the quantity $|g|_X := \inf_{x \in X}d_X(x,gx)$. The isometry $g$ is {\it semisimple} if $|g|_X = d_X(x_0, gx_0)$ for some~$x_0 \in X$. We say $g$ is {\it ballistic} (resp., {\it neutral}) if $|g|_X > 0$ (resp., if $|g|_X=0$), and {\it hyperbolic} if $g$ is both ballistic and semisimple. A subgroup $H < \mathrm{Isom}(X)$ acts {\it neutrally} on $X$ if each $h \in H$ is neutral. 

If $g \in \mathrm{Isom}(X)$ is ballistic, then there is a point $\omega_g \in \partial X$ such that for any $x \in X$, we have $g^nx \rightarrow \omega_g$ as $n \rightarrow \infty$ with respect to the cone topology on $\overline{X}$ \cite{caprace2009isometry}; we call $\omega_g$ the {\it canonical attracting fixed point} of $g$. We use repeatedly the following fact, due to Duchesne \cite[Prop.~6.2]{duchesne2015superrigidity}. For an arbitrary group $G$ and $g_1, \ldots, g_m \in G$, we denote by $\mathcal{Z}_G(g_1, \ldots, g_m)$ the centralizer of $g_1, \ldots, g_m$ in $G$.

\begin{theorem}\label{duchesne}
Let $X$ be a complete $\pi$-visible $\mathrm{CAT}(0)$ space and suppose $g \in \mathrm{Isom}(X)$ is ballistic. Then there is a closed convex subspace $Y \subset X$ and a metric decomposition $Y = Z \times \mathbb{R}$ such that
\begin{itemize}
\item $\mathcal{Z}_{\mathrm{Isom}(X)}(g)$ preserves $Y$ and acts diagonally with respect to the decomposition $Y = Z \times \mathbb{R}$, acting by translations on the second factor;
\item the isometry $g$ acts neutrally on the factor $Z$. 
\end{itemize}
\end{theorem}

In accordance with \cite{bridson1999metric}, we define an isometric action of a group $H$ on a metric space $X$ to be {\it proper} if for any point $x \in X$, there is a neighborhood $U \subset X$ of $x$ such that $\{h \in H \> : \> U \cap h U \neq \infty\}$ is finite. In this case, the set $\{h \in H \> : \> K \cap h K \neq \infty\}$ is finite for any compact subset $K \subset X$ (see, for example, \cite[Remark~I.8.3(1)]{bridson1999metric}). Note, however, that if the metric space $X$ is not proper, then $X$ may contain balls $B$ such that $\{h \in H \> : \> B \cap h B \neq \infty\}$ is infinite; that is, the notion of properness for isometric actions used here is strictly weaker than {\it metric properness}. 

We will make use of the following well-known theorem \cite[Theorem~II.7.1]{bridson1999metric}. 

\begin{theorem}\label{classicalflattorus}
Let $H$ be a free abelian group of rank $r$ acting properly by semisimple isometries on a complete $\mathrm{CAT}(0)$ space $X$. Then $H$ preserves and acts as a lattice of translations on an $r$-dimensional flat in $X$. 
\end{theorem}

\subsection{Thick flats}

A closed convex subspace $Y \subset X$ together with an isometry $\varphi: Y \rightarrow Z \times \mathbb{R}^r$, where $r \geq 0$ and~$Z$ is some complete $\mathrm{CAT}(0)$ space, is called a {\it thick flat} of dimension $r$ in $X$. We say a group~$H$ acting isometrically on~$X$ {\it preserves} the thick flat $(Y, \varphi)$ if $H$ preserves $Y$. Such a group~$H$ {\it acts as a lattice of translations} on the thick flat $(Y, \varphi)$ if $H$ acts diagonally with respect to the decomposition $Z \times \mathbb{R}^r$, acting neutrally on the first factor and by translations on the second, so that the induced map $H \rightarrow \mathbb{R}^r$ embeds $H$ as a lattice of $\mathbb{R}^r$. 

\section{Lemmata}\label{lemmas}
Lemmas \ref{canonical} and \ref{product} are probably well known, but we include their proofs for completeness. The objective is to determine the canonical attracting fixed point of a ballistic isometry acting diagonally on a product.

\begin{lemma}\label{canonical}
Let $Y, Z$ be complete $\mathrm{CAT}(0)$ spaces and $X = Y \times Z$. Suppose $g_Y \in \mathrm{Isom}(Y)$ is neutral and $g_Z \in \mathrm{Isom}(Z)$ is hyperbolic, and let $g, g' \in \mathrm{Isom}(X)$ be the isometries $g_Y \times g_Z$, $\mathrm{Id}_Y \times g_Z$ of $X$, respectively. Then $\omega_g = \omega_{g'}$. 
\end{lemma}

\begin{proof}
There exist a geodesic line $\gamma_Z: \mathbb{R} \rightarrow Z$ in $Z$ and a positive number $\ell$ such that $g_Z(\gamma_Z(t)) = \gamma_Z(t + \ell)$ for any $t \in \mathbb{R}$. The point $\omega_{g'} \in \partial X$ is represented by a geodesic ray of the form~$(y_0, \gamma_Z(t))$, $t \geq 0$, $y_0 \in Y$. Thus, we reduce to the case that $Z = \mathbb{R}$ and $g_Z$ is a translation by~$\ell > 0$. Setting~$x_0 = (y_0, 0)$, we show that the geodesics $\gamma^{(n)}$ in $X$ joining $x_0$ to $g^nx_0$ converge uniformly on compact subsets as $n \rightarrow \infty$ to the geodesic ray $\gamma: [0, \infty) \rightarrow X$ given by $t \mapsto (y_0, t)$. 

To that end, write $\gamma^{(n)}(t) = (\gamma_Y^{(n)}(t), \alpha_n t)$, where $\alpha_n > 0$ and $\gamma^{(n)}_Y$ is a linearly reparameterized geodesic in $Y$ joining $y_0$ to $g_Y^ny_0$, and let $R > 0$. Note that the maximum value of $d_X(\gamma(t), \gamma^{(n)}(t))$ on $[0, R]$ is attained at $t = R$; indeed, for $0 \leq t \leq R$, we have
\begin{equation*}
d_X(\gamma(t), \gamma^{(n)}(t))^2 = d_Y(y_0, \gamma_Y^{(n)}(t))^2 + t^2(1-\alpha_n)^2.
\end{equation*} 
Thus, it suffices to show that $d_X(\gamma(R), \gamma^{(n)}(R)) \rightarrow 0$. This will follow if we can show that $d_Y(y_0, \gamma_Y^{(n)}(R)) \rightarrow 0$ since 
\begin{equation*}
R^2 = d_X(x_0, \gamma^{(n)}(R))^2 = d_Y(y_0, \gamma^{(n)}_Y(R))^2 + \alpha_n^2 R^2.
\end{equation*}
To see that $d_Y(y_0, \gamma_Y^{(n)}(R)) \rightarrow 0$, note that since $\gamma^{(n)}_Y$ is a linearly reparameterized geodesic, we have
\begin{equation*}
\frac{d_Y(y_0, \gamma_Y^{(n)}(R)) }{d_Y(y_0, g_Y^n y_0)} = \frac{R}{d_X(x_0, g^nx_0)}
\end{equation*}
and so
\begin{alignat*}{4}
d_Y(y_0, \gamma_Y^{(n)}(R))^2 &= R^2 \frac{d_Y(y_0, g_Y^ny_0)^2}{d_X(x_0, g^nx_0)^2} \\
&= R^2 \frac{d_Y(y_0, g_Y^ny_0)^2}{d_Y(y_0, g_Y^ny_0)^2 + n^2 \ell^2} \\
&= R^2 \frac{\left( \frac{d_Y(y_0, g_Y^ny_0)}{n} \right)^2}{ \left( \frac{d_Y(y_0, g_Y^ny_0)}{n} \right)^2 + \ell^2}.
\end{alignat*}
Now the latter approaches $0$ as $n \rightarrow 0$ since
\begin{equation*}
\lim_{n\rightarrow \infty}\frac{d_Y(y_0, g_Y^ny_0)}{n} \leq |g_Y|_Y
\end{equation*}
and $|g_Y|_Y = 0$ by assumption. 
\end{proof}

\begin{lemma}\label{product}
Let $X_1, X_2$ be complete $\pi$-visible CAT(0) spaces, let $g_i \in \mathrm{Isom}(X_i)$ for $i=1,2$, and suppose $g_1$ is ballistic. Let $X = X_1 \times X_2$ and let $g = g_1 \times g_2 \in \mathrm{Isom}(X)$. Then $g$ acts ballistically on $X$ and
\[
\omega_g = (\mathrm{arctan}(|g_2|/|g_1|), \omega_{g_1}, \omega_{g_2})
\]
in the spherical join $\partial X_1 * \partial X_2 = \partial X$.
\end{lemma}

\begin{proof}
We suppose first that $g_1, g_2$ are both ballistic, so that we may assume that $X_i $ admits a decomposition $X_i = Y_i \times Z_i$ with respect to which $g_i$ acts diagonally, where $Z_i$ is isometric to~$\mathbb{R}$, and where $g_i$ acts neutrally on the first factor and acts by a translation of $|g_i|$ on the second factor. Let $g_i' \in \mathrm{Isom}(X_i)$ be the product of the identity on $Y_i$ with the translation by $|g_i|$ on $Z_i$, and let $g' = g'_1 \times g'_2 \in \mathrm{Isom}(X)$. Note we have $|g_i| = |g_i'|$, and by Lemma \ref{canonical}, we have $\omega_{g_i} = \omega_{g_i'}$. Moreover, by viewing $X$ as the product $X = (Y_1 \times Y_2) \times (Z_1 \times Z_2)$, we also have $\omega_g = \omega_{g'}$ by Lemma \ref{canonical}. Thus, to establish the lemma, it suffices to show 
\[
\omega_{g'} = (\mathrm{arctan}(|g'_1|/|g'_2|), \omega_{g'_1}, \omega_{g'_2})
\]
but this follows from plane geometry since $g_1', g_2'$ preserve and act as translations on the 2-dimensional flat $\{(y_1, y_2)\} \times (Z_1 \times Z_2) \subset X$, where $y_i$ is any point in $Y_i$. 

If $g_2$ is neutral, then we may only assume that $X_1$ admits a decomposition $X_1 = Y_1 \times Z_1$ as above, and now the lemma follows immediately from Lemma \ref{canonical} by viewing $X$ as the product $X = (Y_1 \times X_2) \times Z_1$.
\end{proof}

We apply Lemma \ref{product} to the special case of matrices acting on symmetric spaces.

\begin{lemma}\label{triangular}
Let $M$ be a symmetric space associated to $\mathrm{GL}_n(\mathbb{C})$ and let $g \in \mathrm{GL}_n(\mathbb{C})$ be of the form
\[
g = \mathrm{diag}(\lambda_1 U_1, \ldots, \lambda_m U_m)
\]
where $\lambda_1, \ldots, \lambda_m \in \mathbb{C}^*$ with $|\lambda_k| \neq 1$ for at least one $k \in \{1, \ldots, m\}$, and $U_k \in \mathrm{SL}_{n_k}(\mathbb{C})$ is an upper unitriangular matrix for $k \in \{1, \ldots, m\}$. Then $g$ acts ballistically on $M$ and has the same canonical attracting fixed point as
\[
g' := \mathrm{diag}(\lambda_1 I_{n_1}, \ldots, \lambda_m I_{n_m})
\]
on $\partial M$. The same statement holds when $\mathrm{GL}_n(\mathbb{C})$ is replaced with $\mathrm{SL}_n(\mathbb{C})$.
\end{lemma}

\begin{proof}
For $k = 1, \ldots, m$, let $X, X_k, Y_k, Z_k$ be the projections of the subgroups
\begin{alignat*}{4}
& \{ \mathrm{diag}(h_1, \ldots, h_m) \> : \> h_k \in \mathrm{GL}_{n_k}(\mathbb{C}) \} \\
& \{ \mathrm{diag}(I_{n_1}, \ldots, I_{n_{k-1}}, \> h, \> I_{n_{k+1}}, \ldots, I_{n_m}) \> : \> h \in \mathrm{GL}_{n_k}(\mathbb{C}) \} \\
& \{ \mathrm{diag}(I_{n_1}, \ldots, I_{n_{k-1}}, \> h, \> I_{n_{k+1}}, \ldots, I_{n_m}) \> : \> h \in \mathrm{SL}_{n_k}(\mathbb{C}) \} \\
& \{ \mathrm{diag}(I_{n_1}, \ldots, I_{n_{k-1}}, \>  e^t I_{n_k} , \> I_{n_{k+1}}, \ldots, I_{n_m}) \> : \> t \in \mathbb{R} \}
\end{alignat*}
of $\mathrm{GL}_n(\mathbb{C})$ to $M$ under the quotient map $\mathrm{GL}_n(\mathbb{C}) \rightarrow M  = \mathrm{GL}_n(\mathbb{C})/\mathrm{U}(n)$, respectively. Then~$X$ is a closed convex subspace of $M$ admitting a decomposition $X = \prod_{k=1}^m X_k$. The subspace~$X_k$ in turn admits a decomposition $X_k = Y_k \times Z_k$, and the factor $Z_k$ is isometric to $\mathbb{R}$. Each of the isometries $g, g'$ preserves $X$ and acts diagonally with respect to the decomposition $X = \prod_{k=1}^m X_k$. On each factor $X_k$, each of $g, g'$ also acts diagonally with respect to the decomposition $X_k = Y_k \times Z_k$, acting neutrally on the first factor and as a translation by  $\alpha_k\ln|\lambda_k|$ on the second for some $\alpha_k > 0$. Thus, the lemma follows from a repeated application of Lemma \ref{product}.

To see that the lemma remains true when $\mathrm{GL}_n(\mathbb{C})$ is replaced with $\mathrm{SL}_n(\mathbb{C})$, note that a symmetric space for $\mathrm{SL}_n(\mathbb{C})$ embeds as a closed convex $\mathrm{SL}_n(\mathbb{C})$-invariant subspace of a symmetric space for $\mathrm{GL}_n(\mathbb{C})$.
\end{proof}

We now observe that a collection of pairwise commuting matrices over $\mathbb{C}$ can be simultaneously put into the form described in Lemma \ref{triangular}.

\begin{lemma}\label{triangularize}
Let $K$ be an algebraically closed field and let $h_\alpha \in \mathrm{M}_n(K)$ be a collection of pairwise commuting matrices. Then there are $s \in \mathbb{N}$ and $C \in \mathrm{SL}_n(K)$ such that
\[
Ch_\alpha C^{-1} = \mathrm{diag}(h_{\alpha,1}, \ldots, h_{\alpha,s})
\]
where $h_{\alpha, \ell} \in \mathrm{M}_{n_\ell}(K)$ is upper triangular and has a single eigenvalue for $\ell = 1, \ldots, s$. 
\end{lemma}

\begin{proof}
Since $K$ is algebraically closed, it suffices to find such $C \in \mathrm{GL}_n(K)$; indeed, we may ultimately replace $C$ with $\mu C$, where $\mu$ is an $n^{\text{th}}$ root of $1/\det(C)$. We now proceed by induction on $n$. The case $n=1$ is trivial. Now let $n > 1$ and suppose the above claim has been established for matrices of smaller dimension. If each of the $h_\alpha$ has a single eigenvalue, then the statement follows from the fact that any collection of pairwise commuting elements of $\mathrm{M}_n(K)$ are simultaneously upper triangularizable \cite[Theorem~1.1.5]{radjavi2012simultaneous}. Now suppose a matrix $h \in \{h_\alpha\}_\alpha$ has more than one eigenvalue. By putting $h$ into Jordan canonical form, for instance, we may assume $h$ is of the form
\[
h = \mathrm{diag}(h_1, h_2)
\]
where $h_i \in \mathrm{M}_{n_i}(K)$ for $i=1,2$ and $h_1, h_2$ do not share an eigenvalue. Since the $h_\alpha$ commute with~$h$, they preserve the generalized eigenspaces of $h$, and so $h_\alpha$ also has a block-diagonal structure
\[
h_\alpha = \mathrm{diag}(h_{\alpha,1}, h_{\alpha,2})
\]
where $h_{\alpha, i} \in \mathrm{M}_{n_i}(K)$ for $i=1,2$. The lemma now follows by applying the induction hypothesis to the collections $\{h_{\alpha,i}\}_\alpha$, $i=1,2$.
\end{proof}

We now prove what one might call a ``thick flat torus theorem."

\begin{lemma}\label{flat}
Suppose $X$ is a complete $\pi$-visible CAT(0) space and $H$ is a free abelian subgroup of $\mathrm{Isom}(X)$ with a basis $h_1, \ldots, h_r \in H$ consisting of ballistic isometries such that for each $m \in \{1, \ldots, r\}$, there is no $(m-1)$-dimensional flat in $X$ whose boundary contains the canonical attracting fixed points $\omega_{h_1}, \ldots, \omega_{h_m}$. Then $H$ preserves and acts as a lattice of translations on a thick flat of dimension $r$ in $X$. 
\end{lemma}

\begin{proof}
We prove by induction the following statement: for $m \in \{1, \ldots, r\}$, there is a closed convex subspace $Y_m$ of $X$ and a decomposition $Y_m = Z_m \times \mathbb{R}^m$ such that 
\begin{itemize}
\item $\mathcal{Z}_{\mathrm{Isom}(X)}(h_1, \ldots, h_m)$ preserves $Y_m$ and acts diagonally with respect to the decomposition $Y_m = Z_m \times \mathbb{R}^m$, acting by translations on the second factor;
\item the subgroup $\langle h_1, \ldots, h_m \rangle$ acts neutrally on the first factor and as a lattice of translations on the second. 
\end{itemize}
The base case $m=1$ is given by Theorem \ref{duchesne}. Now suppose the above holds for $m-1$, where $m \in \{2, \ldots, r\}$. Then $h_m$ must act ballistically on the factor $Z_{m-1}$, since otherwise $\omega_{h_1}, \ldots, \omega_{h_m}$ would be contained in the boundary of $\{z\} \times \mathbb{R}^{m-1}$ by Lemma \ref{canonical}, where $z$ is any point in $Z_{m-1}$. Now $Z_{m-1}$ is a complete $\pi$-visible CAT(0) space, so that by Theorem \ref{duchesne} there is a closed convex subspace $Y$ of $Z_{m-1}$ and a decomposition $Y= Z \times \mathbb{R}$ satisfying
\begin{itemize}
\item $\mathcal{Z}_{\mathrm{Isom}(Z_{m-1})}(h_m)$ preserves $Y$ and acts diagonally with respect to the decomposition $Y = Z \times \mathbb{R}$, acting by translations on the second factor;
\item the action of $h_m$ on the first factor $Z$ is neutral. 
\end{itemize} 
Then the subspace $Y_m := Y \times \mathbb{R}^{m-1} \subset Z_{m-1} \times \mathbb{R}^{m-1}$ has the desired properties.
\end{proof}

The following observation is used in the proof of Lemma \ref{NPC}.

\begin{lemma}\label{conjugate}
Let $X$ be a complete $\mathrm{CAT}(0)$ space and suppose $H < \mathrm{Isom}(X)$ is a free abelian subgroup with a basis $h_1, \ldots, h_r \in H$. Suppose $H$ preserves and acts as a lattice of translations on thick flats $Y, Y'$ in $X$, and let $\phi, \phi'$ be the maps $H \rightarrow \mathbb{R}^r$ induced by the actions of $H$ by translations on the Euclidean factors of $Y, Y'$, respectively. Then the unique linear map ${T: \mathbb{R}^r \rightarrow \mathbb{R}^r}$ satisfying $T(\phi(h_i)) = \phi'(h_i)$ for $i = 1, \ldots, r$ is orthogonal. 
\end{lemma}
\begin{proof}
We wish to show that $T$ preserves the standard inner product on $\mathbb{R}^r$. Since the $\phi(h_i)$ constitute a basis for $\mathbb{R}^r$, it suffices to show that $\langle \phi'(h_i), \phi'(h_j) \rangle = \langle \phi(h_i), \phi(h_j) \rangle$ for $i,j \in \{1, \ldots, r\}$. This is equivalent to saying that for $i,j \in \{1, \ldots, r\}$, we have $\| \phi (h_i) \| = \| \phi'(h_i) \|$ and $\angle (\phi(h_i), \phi(h_j)) = \angle (\phi'(h_i), \phi'(h_j))$. The former is true since
\begin{equation*}
\| \phi (h_i) \| = |h_i|_X = \| \phi' (h_i) \|
\end{equation*}
and the latter is true since $\angle (\phi(h_i), \phi(h_j))$ and $\angle (\phi'(h_i), \phi'(h_j))$ are both equal to the Tits distance between $\omega_{h_i}$ and $\omega_{h_j}$ on $\partial X$ by Lemma \ref{canonical}. 
\end{proof}

The proof of the following lemma borrows heavily from an argument of Leeb; see the proof of Theorem 2.4 in \cite{kapovich1996actions}. Note that we work with the JSJ decomposition of a graph manifold as opposed to its geometric decomposition, so that, for example, the twisted circle bundle over the M\"obius band may appear as a JSJ block of a graph manifold.

\begin{lemma}\label{NPC}
Let $M$ be a graph manifold and suppose $\pi_1(M)$ acts by isometries on a complete $\mathrm{CAT}(0)$ space $X$ such that for each JSJ torus $S$ of $M$, the subgroup $\pi_1(S) < \pi_1(M)$ preserves and acts as a lattice of translations on a thick flat in $X$. Then $M$ admits a nonpositively curved Riemannian metric. 
\end{lemma}
\begin{proof}
Let $B$ be a JSJ block of $M$, and let $f \in \pi_1(B)$ be an element representing a generic fiber of $B$. The element $f$ acts ballistically on $X$ since $f$ is a nontrivial element of $\pi_1(S)$, where $S$ is a torus boundary component of $B$, and $\pi_1(S)$ preserves and acts as a lattice of translations on a thick flat in $X$ by assumption. By Theorem \ref{duchesne}, there is a closed convex subspace $Y \subset X$ with a metric decomposition $Y = Z \times \mathbb{R}$ such that
\begin{itemize}
\item any element of $\pi_1(B)$ preserves $Y$ and acts diagonally with respect to the decomposition $Y = Z \times \mathbb{R}$, acting as a translation on the second factor;
\item the action of $f$ on the first factor $Z$ is neutral.
\end{itemize} 
Moreover, for each element $z \in \pi_1(B)$ representing a boundary component of the base orbifold~$O$ of $B$, the action of $z$ on $Z$ is ballistic since the subgroup $\langle f, z\rangle < \pi_1(B)$ preserves and acts as a lattice of translations on a thick flat in $X$. 

We now realize $B$ as a nonpositively curved Riemannian manifold with totally geodesic flat boundary as follows. Endow the orbifold $O$ with a nonpositively curved Riemannian metric that is flat near the boundary so that the length of each boundary component $c$ of $O$ is equal to the translation length on $Z$ of an element in $\pi_1(B)$ representing $c$. We let $\pi_1(B)$ act on the universal cover $\tilde{O}$ of $O$ via the projection $\pi_1(B) \rightarrow \pi_1(O)$, where $\pi_1(O)$ acts on $\tilde{O}$ by deck transformations. The product of this action with the action of $\pi_1(B)$ on $\mathbb{R}$ coming from the decomposition $Y = Z \times \mathbb{R}$ yields a covering space action of $\pi_1(B)$ on $\tilde{O} \times \mathbb{R}$. The quotient of~$\tilde{O} \times \mathbb{R}$ by this action is the desired geometric realization of $B$. We may do this for each Seifert component of $M$; the flat metrics on any pair of boundary tori that are matched in $M$ will coincide by Lemma \ref{conjugate}, so that we may glue the metrics on the Seifert components to obtain a smooth nonpositively curved metric on $M$.
\end{proof}

The following lemma will not be used in the proofs of Theorems \ref{main} or \ref{maingraph}, but will be applied to derive Corollary \ref{undistortedcor} from Theorem \ref{main}.

\begin{lemma}\label{undistorted}
Let $\Gamma$ be a finitely generated group and $H_0$ a free abelian subgroup of $\Gamma$ of rank~$r \geq 0$. Suppose $\Gamma$ acts on a complete $\mathrm{CAT}(0)$ space $X$ such that $H_0$ preserves and acts as a lattice of translations on a thick flat in $X$. Then $H_0$ is undistorted in $\Gamma$.
\end{lemma}
\begin{proof}
Let $ \mathcal{B} = \{h_1, \ldots, h_r\} \subset H_0$ be a basis for $H_0$, and let $| \cdot |_\mathcal{B}$ be the word metric on $H_0$ with respect to $\mathcal{B}$. Let $\mathcal{S} \subset \Gamma$ be a finite generating set for $\Gamma$ and let $| \cdot |_\mathcal{S}$ be the word metric on $\Gamma$ with respect to $\mathcal{S}$. Let $\phi: H_0 \rightarrow \mathbb{R}^r$ be the homomorphism to $\mathbb{R}^r$ induced by the action of $H_0$ on a thick flat in $X$, and let $y_0 \in Y$,  $K = \max_{s \in \mathcal{S}\cup \mathcal{S}^{-1}}d_X(y_0, sy_0)$. Since any two norms on $\mathbb{R}^r$ are equivalent, there is some $C > 0$ such that $\| \phi(h) \| \geq C |h|_\mathcal{B}$ for any $h \in H_0$. Thus, for $h \in H_0$, we have
\begin{equation*}
K |h|_\mathcal{S} \geq d_X(y_0, hy_0) \geq \| \phi(h) \| \geq C |h|_\mathcal{B}
\end{equation*}
where the first inequality follows from the triangle inequality.
\end{proof}

\section{Proof of Theorems \ref{main} and \ref{maingraph}}\label{proofs}

\begin{proof}[Proof of Theorem \ref{main}]
\begin{enumerate}[leftmargin=0pt, itemindent=20pt,
labelwidth=15pt, labelsep=5pt, listparindent=0.7cm,
align=left, label=(\roman*)]
\item Since $\Gamma$ is finitely generated, we have that $\Gamma \subset \mathrm{SL}_n(A)$ for some finitely generated subdomain $A \subset \mathbb{C}$. Let $E = \mathbb{Q}(A) \subset \mathbb{C}$, so that $E$ is a finitely generated field extension of $\mathbb{Q}$. The extension $E/\mathbb{Q}$ has the structure $\mathbb{Q} \subset F \subset F(T) \subset E$, where $F$ is the algebraic closure of $\mathbb{Q}$ in~$E$, and $T$ is a (possibly empty) transcendence basis for $E$ over $F$. Since the extension $E/\mathbb{Q}$ is finitely generated, the set $T$ is finite and the extensions $F/\mathbb{Q}$ and $E/F(T)$ are of finite degree. 

Let $d = \mathrm{deg}(F/\mathbb{Q})$, and let $\sigma_1, \ldots, \sigma_d$ be the embeddings of $F$ in $\mathbb{C}$. Since $\sigma_j(F)$ is countable but $\mathbb{C}$ is not, the extension $\mathbb{C}/\sigma_j(F)$ has infinite transcendence degree, and hence we may extend~$\sigma_j$ to an embedding $\sigma_j: F(T) \rightarrow \mathbb{C}$. The latter may in turn be extended to an embedding~${\sigma_j: E \rightarrow \mathbb{C}}$ since $E/F(T)$ is algebraic and $\mathbb{C}$ is algebraically closed. The embedding~${\sigma_j: E \rightarrow \mathbb{C}}$ induces an embedding $\sigma_j: \mathrm{SL}_n(E) \rightarrow \mathrm{SL}_n(\mathbb{C})$. Let \[ \sigma: \mathrm{SL}_n(E) \rightarrow G_1 : = \prod_{j=1}^d \mathrm{SL}_n(\mathbb{C}) \] be the diagonal embedding induced by the maps $\sigma_j : \mathrm{SL}_n(E) \rightarrow \mathrm{SL}_n(\mathbb{C})$. Then $\mathrm{SL}_n(E)$ acts by isometries on the Hadamard manifold $X_1 := \prod_{j=1}^d M_j$ via the embedding $\sigma$, where each $M_j$ is a copy of the symmetric space (unique up to scaling of the Riemannian metric) associated to the simple Lie group $\mathrm{SL}_n(\mathbb{C})$. 

By \cite[Prop.~1.2]{alperin1982linear}, there are finitely many discrete valuations $\nu_1, \ldots, \nu_m$ on $E$ such that $A \cap \bigcap_{i=1}^m \mathcal{O}_i \subset \mathcal{O}$, where~$\mathcal{O}$ is the ring of integers of $F$ and $\mathcal{O}_i$ is the valuation ring of $\nu_i$. Let~$B_i$ be the Bruhat–Tits building associated to $\mathrm{SL}_n(E_{\nu_i})$, where $E_{\nu_i}$ is the completion of $E$ with respect to $\nu_i$;  let $X_2 = \prod_{i=1}^m B_i$; and let $\tau: \mathrm{SL}_n(E) \rightarrow G_2 := \prod_{i=1}^m \mathrm{SL}_n(E_{\nu_i})$ be the diagonal embedding. Then $\mathrm{SL}_n(E)$ acts by automorphisms on $X_2$ via the embedding $\tau$. We claim that the diagonal action of $\Gamma$ on ${X := X_1 \times X_2}$ via $\sigma \times \tau: \mathrm{SL}_n(E) \rightarrow G_1 \times G_2$ has the desired properties. 

To that end, let $H$ be a subgroup of $\Gamma$ containing no nontrivial unipotent elements. We first claim that for any vertex $v$ of $X_2$, the subgroup $\sigma(H_v) < G_1$ is discrete, where $H_v$ is the stabilizer of $v$ in $H$. Indeed, let $h \in H_v$. Then for $i=1, \ldots, m$, the element $h$ fixes a vertex of~$B_i$ and (since~$\mathrm{GL}_n(E)$ acts transitively on the vertices of $B_i$) is thus conjugate within $\mathrm{GL}_n(E)$ into~$\mathrm{SL}_n(\mathcal{O}_i)$; in particular, the coefficients of the characteristic polynomial $\chi_h$ of $h$ lie in $\mathcal{O}_i$. Since this is true for each $i \in \{1, \ldots, m\}$ and since $h \in \mathrm{SL}_n(A)$, we have that the coefficients of~$\chi_h$ lie in $A \cap \bigcap_{i=1}^m \mathcal{O}_i$ and hence in $\mathcal{O}$. We thus have a commutative diagram
\begin{equation}\label{diagram}
  \begin{tikzcd}
    G_1 = \prod_{j=1}^d \mathrm{SL}_n(\mathbb{C}) \arrow{r}{P}  & \prod_{j=1}^d \mathbb{C}^n  \\
    H_v \arrow{u}{\sigma} \arrow{r}{p} & \mathcal{O}^n \arrow{u}{\hat{\sigma}}
  \end{tikzcd}
\end{equation}
where the function $p$ maps an element $h \in H_v$ to the $n$-tuple whose entries are the non-leading coefficients of $\chi_h$, the function $P$ is the $d$-fold product of the analogous map $\mathrm{SL}_n(\mathbb{C}) \rightarrow \mathbb{C}^n$, and the function $\hat{\sigma}$ is given by
\[
\hat{\sigma}(\alpha_1, \ldots, \alpha_n) = (\sigma_1(\alpha_1), \ldots, \sigma_1(\alpha_n), \ldots, \sigma_d(\alpha_1), \ldots, \sigma_d(\alpha_n))
\]
for $\alpha_1, \ldots, \alpha_n \in \mathcal{O}$. Since $\hat{\sigma}$ has discrete image (see, for example, Lemma 25.1.10 in \cite{kargapolov1979fundamentals}) and the diagram (\ref{diagram}) is commutative, it follows that $P(\sigma(H_v))$ is discrete in $\prod_{j=1}^d \mathbb{C}^n$. Now suppose we have a sequence $(h_k)_{k \in \mathbb{N}}$ in $H_v$ such that $\sigma(h_k) \rightarrow 1$ in $G_1$. Then, by continuity of the function~$P$, we have $P(\sigma(h_k)) \rightarrow P(1)$. By discreteness of $P(\sigma(H_v))$, this implies that $P(\sigma(h_k)) = P(1)$ for $k$ sufficiently large. It follows that for such $k$ the matrix $h_k$ is unipotent and hence trivial by our assumption that $H$ contains no nontrivial unipotent elements. We conclude that $\sigma(H_v)$ is indeed discrete in $G_1$.  

We now argue that for any $x \in X_2$, there is a neighborhood $V$ of $x$ in $X_2$ such that $H_V \subset H_v$ for some vertex $v$ of $X_2$, where
\[
H_V = \{ h \in H \> : \> V \cap hV \neq \emptyset\}.
\]
Let $c$ be the cell of $X_2$ containing $x$ and let $\ell$ be the dimension of $c$. Let $\epsilon > 0$ be such that the intersection of the ball $B_{X_2}(x, \epsilon)$ with the $\ell$-skeleton $X_2^\ell$ of $X_2$ is contained in $c$. Then we may take $V = B_{X_2}(x, \epsilon/2)$. Indeed, if $h \in H_V$, then $hx \in X_2^\ell \cap B_{X_2}(x, \epsilon) \subset c$, and so $hc = c$. Since~$\mathrm{SL}_n(E)$ acts on $B_i$ without permutations, it follows that $h \in H_v$ for any vertex $v$ of $c$. 

Now, to see that $H$ acts properly on $X$, we observe that for any point $x \in X_2$ and any ball~$B \subset X_1$, the set $U := B \times V \subset X$ has the property that $\{ h \in H \> : \> U \cap hU \neq \emptyset \}$ is finite, where $V \subset X_2$ is as in the preceding paragraph. Indeed, we have $H_V \subset H_v$ for some vertex $v$ of~$X_2$, and $H_v$ acts properly on $X_1$ since $\sigma$ embeds $H_v$ discretely in $G_1$.

\item Suppose $H$ is free abelian with a basis $h_1, \ldots, h_r \in H$. We show that this basis is as in the statement of Lemma \ref{flat}, so that $H$ preserves and acts as a lattice of translations on a thick flat in $X$. Indeed, by Lemma \ref{triangularize}, we may assume that for $j \in \{1, \ldots, d\}$, $k \in \{1, \ldots, r\}$, we have
\[
\sigma_j(h_k) = \mathrm{diag}(h_{j,k,1}, \ldots, h_{j,k, s})
\]
where $h_{j,k,\ell} \in \mathrm{GL}_{n_\ell}(\mathbb{C})$ is upper triangular with a single eigenvalue for $\ell \in \{1, \ldots s\}$. 
We now have a homomorphism $\Delta_j: H \rightarrow \mathrm{SL}_n(\mathbb{C})$ that maps $h \in H$ to the diagonal part of $\sigma_j(h)$; note that $\Delta_j$ is injective since $H$ contains no nontrivial unipotent matrices. The embeddings~$\Delta_j$ produce a diagonal embedding $\Delta: H \rightarrow G_1$. Now let $\Delta': H \rightarrow G_1 \times G_2$ be the product of~$\Delta$ with $\tau \bigr|_H : H \rightarrow G_2$. Then, since $\Delta_j(h)$ has the same characteristic polynomial as $\sigma_j(h)$ for each~$h \in H$, and since $\Delta_j(H)$ contains no nontrivial unipotent matrices, the action of $\Delta'(H)$ on $X$ is proper by the above arguments. Since the latter action is by semisimple isometries, by Theorem \ref{classicalflattorus} there is a genuine $r$-dimensional flat in $X$ preserved by $\Delta'(H)$ on which $\Delta'(H)$ acts as a lattice of translations. Thus, by Lemmas \ref{product} and \ref{triangular}, each nontrivial~$h \in H$ acts ballistically on $X$ and the canonical attracting fixed point of~$h$ on $\partial X$ is equal to that of~$\Delta'(h)$; in particular, $\omega_{h_1}, \ldots, \omega_{h_r}$ must be of the desired form. 

\item Suppose $g \in \Gamma$ is diagonalizable (over $\mathbb{C}$). Since any isometry of $X_2$ is semisimple, to show that $g$ acts as a semisimple isometry of $X$, it suffices to show that $\sigma_j(g)$ is a semisimple isometry of $M_j$ for~$j=1, \ldots, d$. To that end, we show that $\sigma_j(g)$ is diagonalizable. Indeed, since a diagonalization of $g$ has entries in the splitting field $\tilde{E} \subset \mathbb{C}$ of $\chi_g$ over $E$, we in fact have~${g = CDC^{-1}}$ for some $C,D \in \mathrm{SL}_n(\tilde{E})$ with $D$ diagonal (see, for example, \cite[Theorem~8.11]{roman2013advanced}). Since $\mathbb{C}$ is algebraically closed, we may extend $\sigma_j$ to an embedding $\tilde{\sigma}_j: \tilde{E} \rightarrow \mathbb{C}$. Now 
\[
\sigma_j(g) = \tilde{\sigma}_j(g) = \tilde{\sigma}_j(C) \> \tilde{\sigma}_j(D) \> \tilde{\sigma}_j(C)^{-1}
\]
and $\tilde{\sigma}_j(D)$ is diagonal.
\end{enumerate}
 \end{proof}

We recover the following result, due to Button \cite[Theorem~5.2]{button2017properties}.

\begin{corollary}\label{undistortedcor}
Let $\Gamma$ be a finitely generated group and $H$ a distorted finitely generated abelian subgroup of $\Gamma$. Then for any representation $\rho: \Gamma \rightarrow \mathrm{SL}_n(\mathbb{C})$, there is an infinite-order element~${h \in H}$ such that $\rho(h)$ is unipotent. 
\end{corollary}

\begin{proof}
Let $H_0 < H$ be a free abelian subgroup of finite-index, and suppose there is a representation $\rho_0: \Gamma \rightarrow \mathrm{SL}_n(\mathbb{C})$ that does not map any nontrivial element of $H_0$ to a unipotent matrix (in particular, $\rho$ is faithful on $H_0$). Then, by Theorem \ref{main}, there is an action of $\Gamma$ via $\rho$ on a complete CAT(0) space $X$ such that $H_0$ preserves and acts by translations on a thick flat in $X$. By Lemma \ref{undistorted}, it follows that $H_0$ is undistorted in $\Gamma$, and hence the same is true of $H$.
\end{proof}

\begin{proof}[Proof of Theorem \ref{maingraph}]
Suppose otherwise, so that for each JSJ torus $S$ of $M$, the representation~$\rho$ is faithful on $\pi_1(S) < \pi_1(M)$ and the image $\rho(\pi_1(S))$ contains no nontrivial unipotent matrices. Then, by Theorem \ref{main}, there is an action of $\pi_1(M)$ via $\rho$ on a complete CAT(0) space $X$ such that for each JSJ torus $S$ of $M$, the subgroup $\pi_1(S)$ preserves and acts as a lattice of translations on a thick flat in $X$. Thus, $M$ admits a nonpositively curved metric by Lemma \ref{NPC}. 
\end{proof}

\bibliographystyle{amsalpha}
\bibliography{biblio}

\end{document}